\documentclass[12pt,oneside,english]{amsart}
\usepackage[T1]{fontenc}
\usepackage[latin1]{inputenc}
\usepackage{amsthm}
\usepackage{amssymb}
\usepackage{esint}

\makeatletter
\numberwithin{equation}{section}
\numberwithin{figure}{section}

\theoremstyle{plain}
\newtheorem{thm}{Theorem}[section]
  \theoremstyle{definition}
  
  \theoremstyle{plain}

\usepackage{epsfig}\usepackage{mathrsfs}\usepackage{babel}
\numberwithin{equation}{section}

\date{\today}

\newtheorem{lemma}{Lemma}[section]\newtheorem{theorem}[lemma]{Theorem}\theoremstyle{theoreml}\theoremstyle{proposition}\theoremstyle{corollary}\theoremstyle{definition}\newtheorem{definition}[lemma]{Definition}

\newcommand{\beq}{\begin{equation*}}\newcommand{\eeq}{\end{equation*}}\newcommand{\be}{\begin{equation}}\newcommand{\ee}{\end{equation}}\newcommand{\bp}{\begin{proof}}\newcommand{\ep}{\end{proof}}\newcommand{\bd}{\begin{definition}}\newcommand{\ed}{\end{definition}}\newcommand{\bt}{\begin{theorem}}\newcommand{\et}{\end{theorem}}
\newcommand{\Z}{\mathbb{Z}}

\def\ed{\textrm{End}}

\def\cP{{\mathcal P}}

\def\bt{{\bf t}}

\def\chix{{\raise.5ex\hbox{$\chi$}}}

\def\Z{{\mathbb Z}}

\usepackage{babel}

\usepackage{babel}

\usepackage{babel}

\makeatother

\usepackage{babel}
\begin{document}

\author{Lewis Bowen \& Yonatan Gutman}

\title{Corrigendum\\ \vspace{0.5 cm} 
\texttt{L. Bowen. Nonabelian free group actions: Markov processes, the Abramov-
Rohlin formula and Yuzvinskii's formula. Ergodic Theory Dynam. Systems
30 (2010), no. 6, 1629-1663.}}







\maketitle
\section{Introduction}

The paper \cite{Bo10c} proved the Rohlin-Abramov addition formula
and a version of Yuzvinskii's addition formula for actions of free
groups. The proofs have errors. We show here that the proof of the
Rohlin-Abramov formula can be fixed with only a few minor modifications.
We cannot fix the prove of Yuzvinskii's formula in general. However,
in \cite{BG}, we prove it under a mild technical condition for skew
products with compact totally disconnected groups or compact finite
dimensional Lie groups. We conjecture that the formula holds in complete
generality. We also show that Proposition 12.1 of \cite{Bo10c} is
incorrect and obtain a counterexample.
\subsection*{Acknowledgments}
We would like to thank the referee for pointing out a much shorter and elegant counterexample than our original counterexample to Proposition 12.1.

\section{The Abramov-Rohlin addition formula}

\cite[Lemma 9.3]{Bo10c} is incorrect because the support of $\nu$
is not contained in the image of $\phi$ in general. However, the
proof of \cite[Lemma 9.3]{Bo10c} remains correct when $\beta=\alpha^{n}$
(see justification below). This special case is the only case used
to prove \cite[Theorem 9.1]{Bo10c} and the Abramov-Rokhlin addition
formula \cite[Theorem 1.3]{Bo10c}. So those theorems hold as stated.
\begin{proof} {[}Justification of a key step in the proof of Lemma
9.3{]} We now justify the claim that the proof of \cite[Lemma 9.3]{Bo10c}
remains correct when $\beta=\alpha^{n}$. Recall that $K$ is a finite
set and $G=\langle s_{1},\ldots,s_{r}\rangle$ is a finitely generated
free group or semi-group. In the group case, let  $S=\{s_1^{\pm 1}, \ldots, s_r^{\pm 1}\}$. In the semi-group case, let $S=\{s_1,\ldots, s_r\}$. Let $B(e,n)\subset G$ denote the ball of
radius $n$ centered at the identity element (with respect to the
word metric). Let $L=K^{B(e,n)}$ and denote by $\{U_{g}\}_{g\in G}$, respectively $\{T_{g}\}_{g\in G}$,
the $G$-shift action on $L^{G}$, respectively $K^{G}$. Let $\phi:K^{G}\to L^{G}$ be the
map
\[
\phi(x)(g)(f)=x(fg),\quad x\in K^{G},g\in G,f\in B(e,n).
\]
Let $\mu$ be a shift-invariant probability measure on $K^{G}$ and
let $\nu$ be the Markov measure on $L^{G}$ induced from $\phi_{*}\mu$. Denote its support by $\textrm{supp}(\nu)\subset L^G$.  Recall that $\nu$ is the Markov process associated with the invariant transition system $\{P^{s}\}_{s\in S}$, where $P^{s}$ is the stochastic matrix given by $P_{ij}^{s}=\phi_{\ast}\mu(L_{j}^{sg}|L_{i}^{g})$,
$i,j\in L$, $g\in G$, where $L_{i}^{g}=\{z\in L^{G}|\, z(g)=i\}$
(note these quantities do not depend on $g\in G$ by invariance of $\mu$). Fixing $s\in S$, this implies $\nu$ is supported on the collection of all $z\in L^G$ such that if $z \in L_i^g \cap L_j^{sg}$ for some $i,j \in L, g\in G$ then  $P^{s}_{ij}>0$ and therefore $\mu(\phi^{-1}(L^g_i) \cap \phi^{-1}(L^{sg}_j))>0$. So there exists $y=y(z,g,s)\in K^{G}$ with $\phi(y)(g)=z(g)$ and $\phi(y)(sg)=z(sg)$. This is a \textit{local} condition in the sense that $y\in K^{G}$ depends on $g\in G$ and $s\in S$. We will next show that one can find $y_z \in K^G$ satisfying the {\em global} condition $\phi(y_z)=z$. In other words we claim that $\textrm{supp}(\nu)\subset\phi(K^{G})$ (which implies that the proof
of \cite[Lemma 9.3]{Bo10c} remains correct when $\beta=\alpha^{n}$).

To prove the claim, define $\psi:\textrm{supp}(\nu)\to K^{G}$ by $\psi(z)(g)=z(g)(e)$.
It suffices to show that $\phi\psi$ is the identity map on $\textrm{supp}(\nu)$.
Because $\phi$ and $\psi$ are $G$-equivariant, it suffices to prove
that $\phi(\psi(z))(e)=z(e)$ for any $z\in \textrm{supp}(\nu)$. Equivalently, it
suffices to show that for every $f\in B(e,n)$, $\phi(\psi(z))(e)(f)=z(e)(f)$
which, by definition of $\phi$, is equivalent to $\psi(z)(f)=z(e)(f)$.
By definition of $\psi$, this is equivalent to $z(f)(e)=z(e)(f)$.

So let $f\in B(e,n)$. Let us write $f=t_{1}\cdots t_{m}$ where $t_{i}\in S$
and $m\leq n$ is the word length of $f$. Let $f_{i}=t_{i}\cdots t_m$
for $1\le i\le m$. Also let $f_{m+1}=e$ the identity element. Let
$z_{i}$ be the map from $B(f_{i},n)\triangleq B(e,n)f_{i}$, the (left) ball of radius $n$ around $f_{i}$, to $K$ defined by $z_{i}(gf_{i})=z(f_{i})(g)$
for $g\in B(e,n)$. As $f_{i-1}=t_{i-1} f_{i}$, it follows from the definition of $\textrm{supp}(\nu)$ that there exists $y=y(f_i,t_{i-1})\in K^{G}$
so that $z(f_{i})=\phi(y)(f_{i})$ and $z(f_{i-1})=\phi(y)(f_{i-1})$.
Conclude $z_{i}(gf_{i})=y(gf_i)$  and  $z_{i-1}(gf_{i-1})=y(gf_{i-1})$ for $g\in B(e,n)$. We must therefore have that $z_{i}$ and $z_{i-1}$
agree on $B(f_{i},n)\cap B(f_{i-1},n)$ for $2\le i\le m+1$. Therefore,
$z_{1}$ and $z_{m+1}$ agree on the set $\bigcap_{i=1}^{m+1}B(f_{i},n) \ni f$. So
$$z(f)(e)=z_1(f)=z_{m+1}(f)=z(e)(f)$$
as required.

\end{proof}

\section{Proposition 12.1}

The proof of \cite[Proposition 12.1]{Bo10c} relies on the incorrect
\cite[Lemma 9.3]{Bo10c}. Moreover, the statement is incorrect even
when $G=\Z$ because of the next result.

\begin{thm}\label{thm:counter} There exists an ergodic automorphism
$T\in{\rm {Aut}(X,\mu)}$ (where $(X,\mu)$ is a standard probability
space), a finite generating partition $\alpha$ of $X$ and an increasing
sequence $\{\cP_{n}\}_{n=1}^{\infty}$ of finite partitions such that
$\bigvee_{n=1}^{\infty}\cP_{n}$ is the partition into points and
$f_{\mu}(\alpha)=h_{\mu}(T)\ne\liminf_{n\to\infty}H_{\mu}(\cP_{n}|T^{-1}\cP_{n})=\liminf_{n\to\infty}F_{\mu}(\cP_{n})$.
\end{thm}

\begin{proof} Let $X=\{0,1\}^{\Z}$, $\mu$ the uniform probability
on $\{0,1\}$: $\mu(0)=\mu(1)=\frac{1}{2}$, $\mathfrak{B}$ the Borel
$\sigma-$algebra of $X$ and $T:X\rightarrow X$ the shift $(Tx)_{n}=x_{n+1}$.
Let $\mathcal{X}=(X,\mathfrak{B},\mu^{\mathbb{Z}},T)$ be the $(\frac{1}{2},\frac{1}{2})$-Bernoulli
shift. We denote an element $x\in X$ as a sequence $x=\{x_{n}:~n\in\Z\}$.
Let $\alpha=\{C_{0},C_{1}\}$ be a partition where $C_{i}=\{x\in X:~x_{0}=i\}$ for
$i=0,1$. This is a finite generating partition. It is well-known that $h_{\mu}(T)=\log(2)$.

Let $\cP_{n}=\bigvee\{T^{j}\alpha:~|j|\le n,j\ne n-1\}$. Clearly,
$\bigvee_{n=1}^{\infty}\cP_{n}$ is the partition into points. Observe
that $T^{-1}\cP_{n}\vee\cP_{n}=\bigvee\{T^{j}\alpha:~-n-1\le j\le n\}$.
Therefore
\begin{eqnarray*}
H_{\mu}(\cP_{n}|T^{-1}\cP_{n}) & = & H(T^{-1}\cP_{n}\vee\cP_{n})-H(T^{-1}\cP_{n})\\
 & = & (2n+2)\log(2)-2n\log(2)\\
 & = & 2\log(2)>\log(2)=h_{\mu}(T).
\end{eqnarray*}
In particular, $f_{\mu}(\alpha)=h_{\mu}(T)\ne\liminf_{n\to\infty}H_{\mu}(\cP_{n}|T^{-1}\cP_{n})$
as required. \end{proof}

The proof of the addition theorem, \cite[Theorem 13.1]{Bo10c}, relies
on the incorrect \cite[Proposition 12.1]{Bo10c} (however, nothing
else in \cite{Bo10c} relies on this proposition). We conjecture that
the statement of \cite[Theorem 13.1]{Bo10c} is correct. The proof
also relies on \cite[Theorem 13.2]{Bo10c}, a result which is assumed
to follow from minor modifications of \cite[Theorem 2.3]{Th71}. It
now appears that \cite[Theorem 13.2]{Bo10c} does not so follow and
we do not know whether it remains true.






\vskip 0.5cm

\address{Lewis Bowen, Mathematics Department, Mailstop 3368, Texas A\&M University,
College Station, TX 77843-3368 United States.}

\emph{E-mail address:} \texttt{lpbowen@math.tamu.edu} \vskip 0.5cm

\address{Yonatan Gutman, Institut des Hautes Études Scientifiques, Le Bois-Marie,
35 route de Chartres, 91440 Bures-sur-Yvette, France \& Institute
of Mathematics, Polish Academy of Sciences, ul. \'{S}niadeckich 8, 00-956
Warszawa, Poland.}

\emph{E-mail address:} \texttt{yonatan@ihes.fr, y.gutman@impan.pl}
\end{document}